\documentclass[twoside, 12pt]{amsart}
\usepackage[headings]{fullpage}
\usepackage{enumerate}
\usepackage{bookmark}

\newtheorem{thm}{Theorem}
\newtheorem{lem}{Lemma}
\newtheorem{cor}{Corollary}

\renewcommand{\Re}{\mathbb R}
\renewcommand{\epsilon}{\varepsilon}
\newcommand{\Red}{\Re^d}
\newcommand{\Rem}{\Re^m}
\newcommand{\Ze}{\mathbb Z}
\newcommand{\B}{\mathbf B}
\newcommand{\Se}{\mathbb S}
\newcommand{\Sed}{\Se^{d-1}_2}

\newcommand{\st}{\colon}  
\newcommand{\spann}{\operatorname{span}}
\newcommand{\conv}{\operatorname{conv}}
\newcommand{\iprod}[2]{\left< #1,#2\right>}

\title{The {A}lon--{M}ilman {T}heorem for non-symmetric bodies}

\date{}
\author[M. Nasz\'odi]{M\'arton Nasz\'odi}
\thanks{
The author thanks Alexander Litvak and Nicole 
Tomczak-Jaegermann for the discussions on the topic we had some time ago. The 
research was partially supported by the
National Research, Development and Innovation Office (NKFIH) grant 
NKFI-K119670.
}
\address{Department of Geometry, E\"otv\"os Lor\'and University, 
Budapest, Hungary}
\email{marton.naszodi@math.elte.hu}


\begin{document}
\begin{abstract}
A classical theorem of Alon and Milman states that any $d$ dimensional centrally
symmetric convex body has a projection of dimension
$m\geq e^{c\sqrt{\ln{d}}}$ which is either close to the $m$-dimensional 
Euclidean ball or to the $m$-dimensional cross-polytope.
We extended this result to non-symmetric convex bodies.
\end{abstract}

\maketitle

\section{Introduction}

Some fundamental results from the theory of normed spaces have been shown to 
hold in the more general setting of non-symmetric convex bodies. Dvoretzky's 
theorem \cite{Dv61, Mi71} was extended in \cite{LaMa75} and \cite{Gor88}; 
Milman's Quotient of Subspace theorem \cite{Mi85} and duality of entropy 
results were extended in \cite{MP00}. In this note, we extend the Alon--Milman 
Theorem.

A \emph{convex body} is a compact convex set in $\Red$ with non-empty interior. 
We denote the orthogonal projection onto a linear subspace $H$ or $\Red$ by 
$P_H$. For $p=1,2,\infty$, the closed unit ball of $\ell^d_p$ centered at 
the origin is denoted by $\B^d_p$. 
Let $K$ and $L$ be convex bodies in $\Red$ with $L=-L$. We define their 
\emph{distance} as
\[
d(K,L)=\inf\{\lambda>0 \st \ L\subset T(K-a)\subset \lambda L \mbox{ for some } 
a\in \Red \mbox{ and } T\in GL(\Red)\}.
\]
By compactness, this infimum is attained, and when $K=-K$, it is attained with 
$a=0$.

Alon and Milman \cite{AM83} proved the following theorem in the case when $K$ 
is centrally symmetric.
\begin{thm}\label{thm:AMassym}
  For every $\epsilon>0$ there is a constant $C(\epsilon)>0$ with the property 
  that in any dimension $d\in\Ze^+$, and for any convex body $K$ in 
  $\Red$, at least one of the following two statements hold:
  \begin{enumerate}[(i)]
  \item 
    there is an $m$-dimensional linear subspace $H$ of $\Red$ such that
    $d(P_H(K),\B_2^m)<1+\epsilon$, for some $m$ satisfying $\ln\ln m\geq 
    \frac{1}{2}\ln\ln d$, or
  \item 
    there is an $m$-dimensional linear subspace $H$ such that
    $d(P_H(K),\B_1^m)<1+\epsilon$, for some $m$ satisfying $\ln\ln m\geq 
    \frac{1}{2}\ln\ln d - C(\epsilon)$.
  \end{enumerate}
\end{thm}

The main contribution of the present note is a way to deduce 
Theorem~\ref{thm:AMassym} from the original result of Alon and Milman, that is, 
the centrally symmetric case.
By polarity, one immediately obtains

\begin{cor}
  For every $\epsilon>0$ there is a constant $C(\epsilon)>0$ with the property 
  that in any dimension $d\in\Ze^+$, and for any convex body $K$ in 
  $\Red$ containing the origin in its interior, at least one of the following 
two statements hold:
  \begin{enumerate}[(i)]
  \item 
    there is an $m$-dimensional linear subspace $H$ of $\Red$ such 
that
    $d(H\cap K,\B_2^m)<1+\epsilon$, for some $m$ satisfying 
    $\ln\ln m\geq \frac{1}{2}\ln\ln d$, or
  \item 
    there is an $m$-dimensional linear subspace $H$ such that
    $d(H\cap K,\B_{\infty}^m)<1+\epsilon$, for some $m$ satisfying 
    $\ln\ln m\geq \frac{1}{2}\ln\ln d - C(\epsilon)$.
  \end{enumerate}
\end{cor}

\section{Proof of Theorem~\ref{thm:AMassym}}\label{sec:AM}

For a convex body $K$ in $\Red$, we denote its polar by $K^\ast=\{x\in\Red\st 
\iprod{x}{y}\leq 1\mbox{ for all } y\in K\}$. The \emph{support function} of 
$K$ is $h_K(x)=\sup\{\iprod{x}{y}\st y\in K\}$. For basic properties, see 
\cite{Sch14, GiEtAl14}.

First in Lemma~\ref{lem:closetoball}, by a standard argument, we show that if 
the \emph{difference body} $L-L$ of a convex body $L$ is close to the Euclidean 
ball, then so is some linear dimensional section of $L$. For this, we need 
Milman's theorem whose proof (cf. \cite{Mi71, FLM77, MiSch86}) does not use the 
symmetry of $K$ even if it is stated with that assumption.

\begin{lem}[Milman's Theorem]\label{lem:milmann}
  For every $\epsilon>0$ there is a constant $C(\epsilon)>0$ with the property 
  that in any dimension $d\in\Ze^+$, and for any convex body $K$ in 
  $\Red$ with $\B_2^d\subseteq K$, there is an $m$-dimensional 
  linear subspace $H$ of $\Red$ such that
  $(1-\varepsilon)r(\B_2^d\cap H)\subseteq K \subseteq 
(1+\varepsilon)r(\B_2^d\cap H)$, for some $m$ satisfying 
  $m\geq C(\epsilon)M^2d$, where 
  \[M=M(K)=\int\limits_{\Sed} ||x||_K d\sigma(x),\]
  and $r=\frac{1}{M}$.
\end{lem}

\begin{lem}\label{lem:closetoball}
 Let $\alpha,\varepsilon>0$ be given. Then there is a constant 
$c=c(\alpha,\varepsilon)$ with the property that in any dimension $m\in\Ze^+$, 
and for any convex body $L$ in $\Rem$ with $d(L-L,\B_2^m)<1+\alpha$, there is a 
$k$ dimensional linear subspace $F$ of $\Rem$ such that 
$d(P_F(L),\B_2^k)<1+\epsilon$ for some $k\geq cm$.
\end{lem}
\begin{proof}
 Let $\delta=d(L-L,\B_2^m)$. We may assume that 
$\frac{1}{\delta}\B_2^m\subseteq 
L-L\subseteq \B_2^m$. Thus, $h_{L-L}\geq\frac{1}{\delta}$. With the notations 
of 
Lemma~\ref{lem:milmann}, we have
\begin{equation}
M(L^\ast)=\int\limits_{\Sed} ||x||_{L^\ast} d\sigma(x)=
\frac{1}{2}\int\limits_{\Sed} h_L(x)+h_L(-x) d\sigma(x)
\end{equation}
\[
=\frac{1}{2}\int\limits_{\Sed} h_{L-L}(x) d\sigma(x)
\geq\frac{1}{2\delta}\geq\frac{1}{2(1+\alpha)}.
\]

Note that $L^\ast\supset (L-L)^\ast\supset \B^d_2$, thus, by 
Lemma~\ref{lem:milmann} and polarity, we obtain that $L$ has a $k$ dimensional 
projection $P_F$ with $d(P_F L,\B^d_2\cap F)\leq 1+\varepsilon$ and $k\geq 
C(\varepsilon)\frac{1}{4(1+\alpha)^2}m$. Here, $C(\varepsilon)$ is the same as 
in Lemma~\ref{lem:milmann}.
\end{proof}

The novel geometric idea of our proof is the following.
We call a convex body $T=\conv\left( T_1\cup\{\pm e\} \right)$ in $\Re^m$ a 
\emph{double cone} if $T_1=-T_1$ is convex set, $\spann T_1$ is an 
$(m-1)$-dimensional linear subspace, and $e\in\Re^m\setminus\spann T_1$.
Double cones are \emph{irreducible convex bodies}, that is, for any double cone 
$T$, if $T=L-L$ then $L=T/2$, see \cite{NaVi03, Yo91}. We prove a stability 
version of this fact.

\begin{lem}[Stability of irreducibility of double cones]\label{lem:closetocube}
 Let $L$ be a convex body in $\Rem$ with $m\geq 2$, and $T$ be a double cone of 
the form $T=\conv\left( T_1\cup\{\pm e\} \right)$. Assume that $T\subseteq 
L-L\subseteq \delta T$ for some $1\leq \delta<\frac{3}{2}$. Then
\begin{equation*}
\left(\frac{3}{2}-\delta\right)T\subseteq L-a\subseteq 
\left(\delta-\frac{1}{2}\right)T.
\end{equation*}
for some $a\in\Rem$.
\end{lem}
\begin{proof}
 By the assumptions, $e\in T\subseteq L-L$, thus, by translating $L$, we may 
assume that $o, e\in L$.
 Thus, 
\begin{equation}
 L\subseteq (L-L)\cap(L-L+e)\subseteq \delta T \cap (\delta T + e).
\end{equation}
We claim that
\begin{equation}\label{eq:dtcapdt}
 \delta T \cap (\delta T + e) = \frac{e}{2}+\left(\delta-\frac{1}{2}\right)T.
\end{equation}
Indeed, let $H_\lambda$ denote the hyperplane $H_\lambda=\lambda e+ e^{\perp}$. 
To prove \eqref{eq:dtcapdt}, we describe the sections of the right hand side 
and the left hand side by the hyperplanes $H_\lambda$ for all relevant values 
of $\lambda$.
For any $\lambda\in [-\delta,\delta]$, we have
\[\delta T\cap H_\lambda= \delta (T\cap H_{\lambda/\delta})=\lambda e + 
\delta \left(1-\frac{|\lambda|}{\delta }\right)T_1.\]
For any $\lambda\in [-\delta +1,\delta +1]$, we have
\[(\delta T+e)\cap H_\lambda=e+(\delta T\cap H_{\lambda-1})=
\lambda e + \delta \left(1-\frac{|\lambda-1|}{\delta }\right)T_1.
\]
Thus, for any $\lambda\in [-\delta +1,\delta]$, we have
\[\delta T\cap(\delta T+e)\cap H_\lambda=
\lambda e + \delta \left(1-\frac{1}{\delta 
}\max\{|\lambda|,|\lambda-1|\}\right)T_1.\]

On the other hand, for any $\lambda\in [-\delta+1,\delta]$, we have
\[(e/2+(\delta-1/2)T)\cap H_\lambda=\lambda e + 
(\delta-1/2)\left(1-\frac{|\lambda-1/2|}{\delta-1/2}\right)T_1.\]
Combining these two equations yields \eqref{eq:dtcapdt}.

Thus, 
\[
T\subseteq L-L= \left(L-\frac{e}{2}\right) - 
\left(L-\frac{e}{2}\right)\subseteq \left(L-\frac{e}{2}\right)- 
\left(\delta-\frac{1}{2}\right)T.
\]
Using the fact that $T=-T$, and $1\leq \delta<3/2$, we obtain
\[
 \left(\frac{3}{2}-\delta\right)T\subseteq L-\frac{e}{2},
\]
finishing the proof of Lemma~\ref{lem:closetocube}.
\end{proof}

Now, we are ready to prove Theorem~\ref{thm:AMassym}. With the notations of the 
theorem, let $D=K-K$, and apply the symmetric version of the theorem for $D$ in 
place of $K$. We may assume that $\varepsilon<1/2$.
In case (i), we use Lemma~\ref{lem:closetoball} and loose a linear factor in 
the dimension of the almost-euclidean projection.
In case (ii), we use Lemma~\ref{lem:closetocube} with $T=\B^m_1$ and obtain the 
same dimension for the almost-$\ell_1^m$ projection.


\bibliographystyle{alpha}
\bibliography{biblio}

\begin{thebibliography}{BGVV14}

\bibitem[AM83]{AM83}
N.~Alon and V.~D. Milman.
\newblock Embedding of {$l\sp{k}\sb{\infty }$} in finite-dimensional {B}anach
  spaces.
\newblock {\em Israel J. Math.}, 45(4):265--280, 1983.

\bibitem[BGVV14]{GiEtAl14}
Silouanos Brazitikos, Apostolos Giannopoulos, Petros Valettas, and
  Beatrice-Helen Vritsiou.
\newblock {\em Geometry of isotropic convex bodies}, volume 196 of {\em
  Mathematical Surveys and Monographs}.
\newblock American Mathematical Society, Providence, RI, 2014.

\bibitem[Dvo61]{Dv61}
Aryeh Dvoretzky.
\newblock Some results on convex bodies and {B}anach spaces.
\newblock In {\em Proc. {I}nternat. {S}ympos. {L}inear {S}paces ({J}erusalem,
  1960)}, pages 123--160. Jerusalem Academic Press, Jerusalem; Pergamon,
  Oxford, 1961.

\bibitem[FLM77]{FLM77}
T.~Figiel, J.~Lindenstrauss, and V.~D. Milman.
\newblock The dimension of almost spherical sections of convex bodies.
\newblock {\em Acta Math.}, 139(1-2):53--94, 1977.

\bibitem[Gor88]{Gor88}
Yehoram Gordon.
\newblock Gaussian processes and almost spherical sections of convex bodies.
\newblock {\em Ann. Probab.}, 16(1):180--188, 1988.

\bibitem[LM75]{LaMa75}
D.~G. Larman and P.~Mani.
\newblock Almost ellipsoidal sections and projections of convex bodies.
\newblock {\em Math. Proc. Cambridge Philos. Soc.}, 77:529--546, 1975.

\bibitem[Mil71]{Mi71}
V.~D. Milman.
\newblock A new proof of {A}. {D}voretzky's theorem on cross-sections of convex
  bodies.
\newblock {\em Funkcional. Anal. i Prilo\v zen.}, 5(4):28--37, 1971.

\bibitem[Mil85]{Mi85}
V.~D. Milman.
\newblock Almost {E}uclidean quotient spaces of subspaces of a
  finite-dimensional normed space.
\newblock {\em Proc. Amer. Math. Soc.}, 94(3):445--449, 1985.

\bibitem[MP00]{MP00}
V.~D. Milman and A.~Pajor.
\newblock Entropy and asymptotic geometry of non-symmetric convex bodies.
\newblock {\em Adv. Math.}, 152(2):314--335, 2000.

\bibitem[MS86]{MiSch86}
Vitali~D. Milman and Gideon Schechtman.
\newblock {\em Asymptotic theory of finite-dimensional normed spaces}, volume
  1200 of {\em Lecture Notes in Mathematics}.
\newblock Springer-Verlag, Berlin, 1986.
\newblock With an appendix by M. Gromov.

\bibitem[NV03]{NaVi03}
M{\'a}rton Nasz{\'o}di and Bal{\'a}zs Visy.
\newblock Sets with a unique extension to a set of constant width.
\newblock In {\em Discrete geometry}, volume 253 of {\em Monogr. Textbooks Pure
  Appl. Math.}, pages 373--380. Dekker, New York, 2003.

\bibitem[Sch14]{Sch14}
Rolf Schneider.
\newblock {\em Convex bodies: the {B}runn-{M}inkowski theory}, volume 151 of
  {\em Encyclopedia of Mathematics and its Applications}.
\newblock Cambridge University Press, Cambridge, expanded edition, 2014.

\bibitem[Yos91]{Yo91}
David Yost.
\newblock Irreducible convex sets.
\newblock {\em Mathematika}, 38(1):134--155, 1991.

\end{thebibliography}

\end{document}